\documentclass[11pt,reqno]{article}
\usepackage[utf8]{inputenc}
\usepackage[margin=2.8cm]{geometry}
\usepackage{macros}
\newcommand{\dk}{\mathrm{d}_\mathrm{K}}
\newcommand{\wk}{\mathrm{w}_{\mathrm{K}}}

\usepackage{setspace}
\title{\textbf{$t$-Balanced Codes with the Kendall-$\tau$ Metric}}
\author{Benjamin Jany, Alberto Ravagnani}
\date{}

\begin{document}

\maketitle

\begin{abstract}
    We investigate the maximum cardinality and the mathematical structure of error-correcting codes endowed with the Kendall-$\tau$ metric. We establish an averaging bound for the cardinality of a code with prescribed minimum distance, discuss its sharpness, and characterize codes attaining it. This leads to introducing the family of $t$-balanced codes in the 
    Kendall-$\tau$ metric. The results are based on novel arguments that shed new light on the structure of the Kendall-$\tau$ metric space.
\end{abstract}


\bigskip

\section{Introduction}
The Kendall-$\tau$ distance is a measure for the discrepancy between two permutations of the set $\{1, \ldots,n\}$. It counts the number of transpositions needed to transform one permutation into another.

The Kendall-$\tau$ distance naturally arises in various contexts within applied mathematics. For instance, it was introduced in statistics to measure the similarity of data samples, in the context of rank correlation~\cite{forthofer1981rank}.
In coding theory, codes endowed with the 
Kendall-$\tau$ distance
were proposed for rank modulation schemes, which are relevant 
in the context of flash memories
~\cite{chadwick1969rank,barg2010codes}.
Applications in iterative voting have been considered in~\cite{farnoudconsensus13,koolyk2016convergence}.
In these contexts, codes are collection of permutations where the Kendall-$\tau$ between each two of them is bounded from below by a given integer.

This paper focuses on the structural properties of the Kendall-$\tau$ space and on the properties of its error-correcting codes. In the literature, two major, complementary research directions can be found in this context: 1) establishing upper bounds for the cardinality of a code with prescribed minimum Kendall-$\tau$ distance; and 2) 
constructing large codes with 
prescribed minimum Kendall-$\tau$ distance. 
For instance, several bounds were established in \cite{barg2010codes}, 
which also contains 
a proof of the existence of codes that correct a constant number of errors and have size within a constant factor of the sphere packing bound. 
Bounds are derived also in~\cite{kong2012comparing,sarit,jiangcorrecting10, nguyen2024improving,vijayakumaran2016largest, buzaglo2015bounds, wang2021nonexistence}. Constructions of error-correcting codes in the Kendall-$\tau$ metric were presented  in \cite{mazumdar2012constructions}, with
an asymptotic analysis and a construction of families of rank modulation codes that can correct a number of errors that grow with $n$ at varying rates, from $\Theta(n)$ to $\Theta(n^2)$. 
In \cite{zhou2015}, the authors construct systematic error-correcting codes and more constructions are given in \cite{zhang2015snake}.

In this paper, we bring forward the structural theory of codes endowed with the 
Kendall-$\tau$ metric. We start by surveying the main properties of the 
Kendall-$\tau$ distance, which we need throughout the paper. We then establish a new bound for codes with the 
Kendall-$\tau$ metric, based on an averaging argument, which we call the \textit{averaging bound} (see Section~\ref{sec:averaging}). We compare the bound with the state of the art and introduce the class of $t$-balanced codes (those meeting the averaging bound with equality). We then discuss the existence of 
$t$-balanced codes and characterize them for the parameters where they exist (Section~\ref{sec:exist}). 

\paragraph*{Acknowledgement.}
Benjamin Jany is supported by the Dutch Research Council via grant VI.Vidi.203.045.
Alberto Ravagnani is supported by the Dutch Research Council via grants VI.Vidi.203.045 and OCENW.KLEIN.539, and by the Royal Academy of
Arts and Sciences of the Netherlands.

\section{Codes for the Kendall-$\tau$ Metric}

In this section, we recall basic definitions and properties of the Kendall-$\tau$ metric and of codes endowed with that metric. We start by establishing the notation for this paper.

\begin{notation}
We denote by $\mbS_n$ the group of permutations of $n$ elements. The identity element of~$\mbS_n$ is denoted as $\epsilon$. We use both single line and cycle notation (with fixed points omitted). The single line notation will be written with square brackets, i.e.,  $[ \sigma(1) \, \sigma(2) \, \ldots \sigma(n)]$, whereas the cycle notation will be written with round brackets. For all $\sigma, \tau \in \mbS_n$, the product of permutation is denoted by $\sigma  \tau$ and is defined by $(\sigma  \tau) (i) = \tau(\sigma(i))$ for all $1 \leq i \leq n$.
\end{notation} 

\begin{definition}
The \emph{Kendall-$\tau$ distance} between $\sigma, \tau \in \mbS_n$, denoted $\dk(\sigma, \tau)$, is the minimum number of consecutive transposition required to obtain $\sigma$ from $\tau$. 
\end{definition}

It is well known that the Kendall-$\tau$ distance can be characterized as follows;
we refer to~\cite{jiangcorrecting10} for the proof. 

\begin{theorem}\label{thm:distequiv}
Let $\sigma, \tau \in \mbS_n$. We have 
$$\dk(\sigma, \tau) = \left| \{(i,j) \, : \, i \neq j, \,  \sigma^{-1}(i)< \sigma^{-1}(j), \, \tau^{-1}(i) > \tau^{-1}(j)\}\right|. $$
\end{theorem}

The Kendall-$\tau$ metric is right invariant, meaning that for all $\sigma, \tau, \pi \in \mbS_n$ we have \begin{equation} \label{RI}
    \dk(\sigma, \tau) = \dk(\sigma \alpha, \tau \alpha).
\end{equation} A proof of this fact can be found  \cite{huangmetrics}.

\begin{definition}
A  \emph{code} is a subset $\mC \subseteq \mbS_n$ with 
$|\mC|\ge 2$ and its elements are called \emph{codewords}. In this paper, codes are endowed with the 
Kendall-$\tau$ metric. The \emph{minimum distance} of a code $\mC$ is the integer $\dk(\mC) := \min\{\dk(\sigma, \tau) \,: \sigma, \tau \in \mC \textup{ and } \sigma \neq \tau\}$.
For all $\sigma \in \mbS_n$, the \emph{weight} of $\sigma$ is  $\wk(\sigma) := \dk(\sigma, \epsilon) = \dk(\epsilon, \sigma)$. 
 The inversions of $\sigma \in \mbS_n$ are the elements of the set $I_{\sigma} := \{ (i,j) \in [n] \, : \, i > j \textup{ and } \sigma^{-1}(i) < \sigma^{-1}(j)\}$.
Using \cref{thm:distequiv}, it can be seen that the weight of a permutation is equal to the number of inversion of said permutation. i.e $\wk(\sigma) = |I_{\sigma}|$ for all $\sigma \in \mbS_n$.
\end{definition}

The Kendall-$\tau$ distance between permutations can be computed by only looking at their inversions, as described in the following result. 

\begin{proposition}\label{prop:distsupp}
Let $\sigma, \tau \in \mbS_n$. We have 
$$\dk(\sigma, \tau) = \left| (I_{\sigma} \cup I_{\tau}) \setminus (I_{\sigma} \cap I_{\tau}) \right|.$$
\end{proposition}

\begin{proof}
Let $I = (I_{\sigma} \cup I_{\tau}) \setminus (I_{\sigma} \cap I_{\tau})$ and $I':= \{(i,j) \, : \, i \neq j \,, \,  \sigma^{-1}(i)< \sigma^{-1}(j) \, \textup{ and } \tau^{-1}(i) > \tau^{-1}(j)\}$. We will show that $|I| = |I'|$. Let $(i,j) \in I$. If $(i,j) \in I_{\sigma} \setminus I_{\tau}$, then $\sigma^{-1}(i) < \sigma^{-1}(j)$ but $\tau^{-1}(i) < \tau^{-1}(j)$, hence $(i,j) \in I'$. On the other hand, if $(i,j) \in I_{\tau} \setminus I_{\sigma}$, then  $\sigma^{-1}(j) < \sigma^{-1}(j)$ but $\tau^{-1}(i) > \tau^{-1}(i)$, hence $(j,i) \in I'$. Therefore $|I| \le |I'|$. Now let $(i,j) \in I'$. If $i>j$ then $(i,j) \in I_{\sigma} \setminus I_{\tau}$. Whereas if $i < j$ then $(j,i) \in I_{\tau} \setminus I_{\sigma}$. Hence $|I'| \le |I|$, which concludes the proof. 
\end{proof}

A central, unsolved problem in the study of codes with the Kendall-$\tau$ metric is to determine the maximum cardinality of a code with given minimum distance. 
Various bounds for such cardinality have been established and we list a few of them below. In the first two bounds, we denote by $B_{\rmK}(n,\ell)$ be the ball of radius $\ell$ in $\mbS_n$ under the Kendall-$\tau$ metric:

$$B_{\rmK}(n , \ell)=\{\tau \in \mbS_n \, : \, \dk(\epsilon,\tau) \le \ell\}.$$

\begin{itemize}
    \item[(i)] \textbf{Sphere Packing bound}; see \cite[Thm 1]{sarit}. Let $\mC \subseteq \mbS_n$ be a code of minimum distance $d$. We have 
\begin{equation}\label{eq:spherepacking} 
|\mC| \leq \frac{n!}{|B_{\rmK}(n,\lfloor (d-1)/2 \rfloor)|}.
\end{equation}

\item[(ii)] \textbf{Gilbert-Varshamov-like bound}; see \cite[Thm 13]{jiangcorrecting10}.
 Let $n,M,d$ be positive integers. If
\begin{equation}M \leq \frac{n!}{|B_K(n,d-1)|},\end{equation}
then there exist a code $\mC \subseteq \mbS_n$ of minimum distance $d$ with $|\mC| = M$.
  
\item[(iii)] \textbf{Singleton-type bound}; see \cite[Thm 14]{jiangcorrecting10}. Let $\mC \subseteq \mbS_n$ with minimum distance $d$ and let $1 \leq t \leq n-2$ be an integer. \begin{itemize}
    \item[(1)] If $t$ is the smallest integer with $|\mC| > \frac{n!}{t!}$, then 
    \begin{equation}\label{eq:singl1}
    \dk(\mC) \leq \binom{t}{2}.\end{equation}
    \item[(2)] If $|\mC| = \frac{n!}{t!}$, then
    \begin{equation}\label{eq:singl2}
    \dk(\mC)  \leq \binom{t}{2} +1.\end{equation}
\end{itemize}
\end{itemize}

Different methods can be used to derive the bounds above. For example, \eqref{eq:spherepacking} can be established from a code-anticode approach, whereas \eqref{eq:singl1} and \eqref{eq:singl2} were derived by considering arrays of permutations.
In this paper we derive a new bound using an averaging argument and the notion of puncturing.

\section{Averaging Bound and $t$-Balanced Codes} \label{sec:averaging}

In this section we establish a new bound on the maximum cardinality of a code under the Kendall-$\tau$ metric, given the length of the code and its minimum distance. 
To do so we use an averaging argument on the number of codewords with a given first entry. We will furthermore need the operation of puncturing (introduced in \cite{zhou2015}) which we recall below. 

\begin{definition}
For a set $S \subseteq [n]$ with $|S|=s$ and $\sigma \in \mbS_n$. Label the elements of $S$ as $i_1 <  \cdots < i_s$. Let 
 $\sigma|_S \in \mbS_s$ be the permutation such that for all $j, l \in [s]$, $\sigma|_S(j) < \sigma|_S(l)$ if and only if $\sigma(i_j) < \sigma(i_l)$; see for instance \cref{ex:punct}.
We call 
$\sigma|_S$ the $S$-\emph{puncturing} of $\sigma$. 
   For a code $\mC \subseteq \mbS_n$ and $S \subseteq [n]$,  the $S$-\emph{puncturing} of $\mC$ is the code $\mC|_S = \{\sigma|_S \, : \, \sigma \in \mC\} \subseteq \mbS_s$, where $s=|S|$.
\end{definition}

\begin{example}\label{ex:punct}
    Let $\sigma := \begin{bmatrix}6 & 1 & 3 & 5 & 2& 4\end{bmatrix}\in \mbS_6$ and $S = \{3, 5, 6\}$. Then  $\sigma|_S = \begin{bmatrix}2 &1 &3\end{bmatrix}$.
\end{example}

We now investigate how the distance between permutations changes after puncturing. We start by introducing some symbols.

\begin{notation}
    For $\sigma \in \mbS_n$, let 
    \begin{itemize}
        \item $I_{\sigma, <i} := \{(i,j) \, : (i,j) \in I_{\sigma}\}$,
        \item $I_{\sigma, >i} := \{(j,i) \, : (j,i) \in I_{\sigma}\}$,
        \item $I_{\sigma, i} = I_{\sigma, <i} \cup I_{\sigma, >i}$,
        \item $I_{\sigma, \neq i} := I_{\sigma} \setminus S_{\sigma, i}$.
    \end{itemize}
Furthermore let $\psi_i: [n]\setminus \{i\} \rightarrow [n-1]$ be the map defined by 
$$\psi_i(j) = \begin{cases}
    j & \textup{ if } j < i \\
    j-1 & \textup{ if } j > i.
\end{cases}$$
\end{notation}

We now relate the number of inversions of an $S$-puncturing in terms of its original number of inversions, where $S$ is the complement of a singleton.

\begin{lemma}\label{lem:supportpuncture}
    Let $\sigma \in \mbS_n$, $i  \in [n]$, and $S_i := [n] \setminus \{i\}$. We have 
$$|I_{\sigma|_{S_i}}| = |I_{\sigma}| - |I_{\sigma, \sigma(i)}| = |I_{\sigma, \neq \sigma(i)}|.$$
\end{lemma}

\begin{proof}
    Let $\sigma' := \sigma|_{S_i}$ and $j:= \sigma(i)$. By definition,  $(s, \ell) \in I_{\sigma, \neq j}$  if and only if $s \neq j \neq \ell$, $s > \ell$ and $\sigma^{-1}(s) < \sigma^{-1}(\ell)$ if and only if $\psi_j(s) > \psi_j(\ell)$ and $(\sigma')^{-1}(\psi_j(s)) < (\sigma')^{-1}(\psi_j(\ell))$ if and only if $(\psi_j(s), \psi_j(\ell)) \in I_{\sigma'}$. Because the map $\psi_j$ is bijective it follows that 
    $|I_{\sigma} \backslash I_{\sigma, j}| = |I_{\sigma, \neq j}| = |I_{\sigma'}|$. The second equality of the statement follows by definition.    
\end{proof}

The next two results describe how the weight of a permutation and the distance between permutations (respectively) behave under some puncturing operations.

\begin{proposition}\label{prop:weightpunct}
    Let $S_i = [n] \setminus \{i\}$ for $i \in \{1,n\}$ and let $\sigma \in \mbS_n$ such that $\sigma(i) = j$. We have 
$$\wk(\sigma|_{S_i}) = \begin{cases} \wk(\sigma) - (j-1) & \textrm{if } i =1, \\
                                      \wk(\sigma) - (n-j)& \textrm{if } i=n.     
\end{cases}$$
\end{proposition}

\begin{proof}
 By \cref{lem:supportpuncture} we have $\smash{|I_{\sigma|_{S_i}}| = |I_{\sigma}| - |I_{\sigma, j}|}$ for $i \in \{1, n\}$.
 If  $i =1$, since $\sigma^{-1}(j) = 1$ we have $\smash{I_{\sigma,j} = \{ (j, i) \, : 1 \leq i < j\}}$. Thus $\smash{|I_{\sigma,j}| = (j-1)}$ and $\smash{\wk(\sigma|_{S_1}) = |I_{\sigma|_{S_1}}| = \wk(\sigma) - (j-1)}$.
 
If $i = n$, since $\sigma^{-1}(j) = n$ we have $I_{\sigma,j} = \{ (i, j) \, : j <   i \leq n\}$. Therefore $|I_{\sigma,j}| = n- i$  and $\wk(\sigma|_{S_n}) = \wk(\sigma) - (n-j).$
\end{proof}

\begin{proposition}\label{prop:distpuncS1}
    Let  $\sigma, \tau \in \mbS_n$ and $i \in \{1, n\}$. If $\sigma(i) = \tau(i)$, then 
    $$\dk(\sigma|_{S_i}, \tau|_{S_i}) =  \dk(\sigma , \tau). $$
\end{proposition}

\begin{proof}
We start with $i = 1$.
    Let $j = \sigma(1) = \tau(1)$,  $\sigma' = \sigma|_{S_i}$, and $\tau' =  \tau|_{S_i}$. By \cref{prop:distsupp},
    \begin{align*}
        \dk(\sigma, \tau) &= |I_{\sigma} \cup I_{\tau}| - |I_{\sigma} \cap I_{\tau}|\\
        &= |I_{\sigma, \neq j} \cup I_{\tau, \neq j}| - |I_{\sigma, \neq j} \cap I_{\tau, \neq j}| + |I_{\sigma,  j} \cup I_{\tau, j}| - |I_{\sigma, j} \cup I_{\tau, j}|\\
        &=  |I_{\sigma, \neq j}| +  |I_{\tau, \neq j}| - 2|I_{\sigma, \neq j} \cap I_{\tau, \neq j}| + |I_{\sigma,  j} \cup I_{\tau, j}| - |I_{\sigma, j} \cup I_{\tau, j}|.
    \end{align*}
    Note that  $(s, \ell) \in I_{\sigma, \neq j} \cap I_{\tau, \neq j}$ if and only if $(\psi_j(s), \psi_j(\ell)) \in I_{\sigma'} \cap I_{\tau'}$. In addition, since $\psi_j$ is bijective, $\smash{|I_{\sigma, \neq j} \cap I_{\tau, \neq j}| = |I_{\sigma'} \cap I_{\tau'}|}$. 
    Therefore, together with \cref{lem:supportpuncture}, we get
    \begin{align*}
        \dk(\sigma, \tau) &=  |I_{\sigma'}| +|I_{\tau'}| - 2|I_{\sigma'} \cap I_{\tau'}| + |I_{\sigma, j } \cup I_{\tau, j}| - |I_{\sigma, j} \cap I_{\tau, j}|\\
        &= |I_{\sigma'} \cup I_{\tau'}| - |I_{\sigma'} \cap I_{\tau'}| + |I_{\sigma, j } \cup I_{\tau, j}| - |I_{\sigma, j} \cap I_{\tau, j}|\\
        &= \dk(\sigma' , \tau') + |I_{\sigma, j } \cup I_{\tau, j}| - |I_{\sigma, j} \cap I_{\tau, j}|.
    \end{align*} 
Finally, since $\sigma(1) = j$ we have $(j, \ell) \in I_{\sigma, j}$ if and only if $\ell < j$. Because the same holds for $I_{\tau, j}$, we have $I_{\sigma, j} = I_{\tau, j}$ and therefore $\dk(\sigma, \tau) =  \dk(\sigma' , \tau')$.
The proof for the case $i = n$ is similar and therefore omitted. 
\end{proof}

We are ready to present our averaging bound for codes with the Kendall-$\tau$ metric, which is the main result of this section. Its proof relies on \cref{prop:distpuncS1}. In the next section we will
classify the codes whose parameters meet the bound with equality. 

\begin{notation}\label{nott}
    In the sequel, for $\mC \subseteq \mbS_n$ we let $\mC^{i,j} := \{\sigma \in \mC \, : \, \sigma(i) = j\}$, where $1 \leq i, j \leq n$.
\end{notation}

\begin{theorem}\label{thm:Singletonbound2}
    Let $\mC \subseteq \mbS_n$ and suppose that $\dk(\mC) > \binom{t}{2} $ for some $t \in [n]$. Then $|\mC| \leq \frac{n!}{t!}$.
\end{theorem}

\begin{proof}
For notation purposes, let $\mC^{i} := \mC^{1,i}$ and $S_i = [n] \setminus \{i\}$. 
Using an averaging argument, one can easily see that $|\mC^i| \geq |\mC|/n$ for some $1 \leq i \leq n$. For that $i$, we have $|\mC^i|_{S_1}| = |\mC^i|$ and  $\dk(\mC^i|_{S_1}) \geq \dk(\mC)$, where the latter follows from \cref{prop:distpuncS1}.

Let $\mC_1 = \mC^i|_{S_1} \subseteq \mbS_{n-1}$. 
Using an averaging argument as before, we see that there exists $1 \leq j \leq n-1$ with $|\mC_1^j| \geq |\mC_1|/(n-1)$.
Consider the code $\mC_2 := \mC_1^j|_{S_1} \subseteq \mbS_{n-2}$. We have $|\mC_2|= |\mC_1^j|$ and  $\dk(\mC_2) \geq \dk(\mC_1^j)$, again by \cref{prop:distpuncS1}.
Iterating this process $n-t$ times, we arrive to a code $\mC_{n-t} \subseteq \mbS_{t}$ with $|\mC_{n-t}| \geq \frac{|\mC|}{n(n-1)\cdots(t+1)}$ and $\dk(\mC_{n-t}) \leq \dk(\mC_{n-t+1})$.

Using the assumption on the minimum distance of $\mC$, we furthermore have that $|\mC_{n-t}| \leq 1$. Indeed, if $|\mC_{n-t}| \geq 2$ then $\dk(\mC_{n-t}) \leq \binom{t}{2}$. However, $\dk(\mC) \leq \dk(\mC_1) \leq \cdots \leq \dk(\mC_{n-t}) \leq \binom{t}{2}$, which leads to a contradiction on the minimum distance of $\mC$. 
 Hence we get that 
    $$\frac{|\mC|}{n(n-1) \ldots (t+1)} \leq |\mC_{n-t}| \leq 1,$$
concluding the proof.
\end{proof}

The above bound naturally gives rise to the following concept.

\begin{definition}\label{def:TMDS}
Let $1 \leq t \leq n$ be an integer. We say that a code $\mC \subseteq \mbS_n$ is \emph{$t$-balanced} if $\dk(\mC) > \binom{t}{2}$ and $|\mC| = \frac{n!}{t!}$.
\end{definition}


Note, a similar bound can be found in~\cite[Thm 14]{jiangcorrecting10}; see \eqref{eq:singl2}. However, in~\cite{jiangcorrecting10} the authors fix the cardinality of the code and derive a bound on the minimum distance, whereas we fix the minimum distance and derive a bound on the code's cardinality. Although our bound could also be considered a Singleton-type bound, we chose to name it averaging bound to distinguish it from the bounds in~\cite[Thm 14]{jiangcorrecting10}.
Furthermore, in \cite{jiangcorrecting10} the authors define \emph{MDS codes} as being the codes whose minimum distance achieves their bound with equality. As we show in the next result, our notion of $t$-balanced only agrees with the notion of MDS code from \cite{jiangcorrecting10} in the case where $|\mC| = n!/t!$ and $\smash{\dk(\mC) = \binom{t}{2} +1}$.

\begin{corollary}\label{cor:optmindist}
    If $\mC \subseteq \mbS_n$ is a $t$-balanced code, then $\dk(\mC) = \binom{t}{2} + 1$. In particular $t$-balanced codes are MDS codes. 
\end{corollary}

\begin{proof}
Let $\mC \subseteq \mbS_n$ be a $t$-balanced code. Since $|\mC| = \frac{n!}{t!}$, by \eqref{eq:singl2} we have $\dk(\mC) \leq \binom{t}{2} +1$. However, we also know, by assumption, that $\dk(\mC) \geq \binom{t}{2} +1$. Thus equality must hold true. The second statement follows from the definition of MDS codes established in \cite{jiangcorrecting10}.
\end{proof}

\begin{remark}
    \cref{thm:Singletonbound2} can also be derived from the code-anticode bound established in \cite[Cor.~1]{sarit}. In fact, the set $\smash{\mcA_{n-t} = \{ \sigma \in \mbS_n \, : \, \sigma(i) = i \textrm{ for all } t+1 \leq i \leq n\}}$ is an anticode of cardinality~$t!$ and diameter $\smash{\binom{t}{2}}$. We call such an anticode a \emph{cube} of \emph{length} $t$. 
    It is unclear whether the cardinality of a cube of length $t$ is ever greater than a ball of radius $\binom{t}{2}/2$. If the latter is ever the case, then \cref{thm:Singletonbound2} would be tighter than \eqref{eq:spherepacking} for some choices of parameters.
    However, at the time of writing this paper, no 
    explicit or easy-to-evaluate formula 
    exists to determine the cardinality of a  Kendall-$\tau$ ball of given radius. For small enough values of $t$, computations done using a computer system seem to suggest that the ball of radius $\binom{t}{2}/2$ is larger than cubes of length $t$ for $t>2$, but we were unable to prove this fact. We leave this as an open problem.
\end{remark}

\cref{def:TMDS} suggests the following question: Do  $t$-balanced codes exist for $t \geq 2$?
The even subgroup $A_n \leq \mbS_n$, consisting of all permutations of even weight, has cardinality $n!/2$ and minimum distance $2$. Therefore, it is a $2$-balanced code. The next natural questions would be: 1) Can we classify $2$-balanced codes? \ 2) Do there exist $t$-balanced codes for $t \geq 3$? The remaining of this paper will answer both questions.

\section{Existence and Classification}
\label{sec:exist}

In this section we answer two previously asked questions: 1) Can we classify $2$-balanced codes? \ 2) Do there exist $t$-balanced codes for $t \geq 3$?
We start by determining some combinatorial properties $t$-balanced codes must satisfy. Recall~\cref{nott}.

\begin{lemma}\label{lem:optpunct}
    Let $\mC \subseteq \mbS_n$ be a $t$-balanced code and let $S_j = [n] \setminus \{ j\}$ for $j \in [n]$. For all $1 \leq i \leq n$ we have 
    $$|\mC^{1,i}|_{S_1}| =  \frac{|\mC|}{n} =  |\mC^{n,i}|_{S_n}|.$$
In particular, if $\mC$ is $t$-balanced then $\mC^{i,j}|_{S_i} \subseteq \mbS_{n-1}$ is also $t$-balanced for all $i \in \{1, n\}$ and $j \in [n]$.
\end{lemma}

\begin{proof}
    We show the statement for $\mC^{1,i}|_{S_1}$. Using the symmetry of permutations, one can easily derive a similar argument to show the desired result for $\mC^{n,i}|_{S_n}$.  To simplify the notation, let $\mC^i := \mC^{1,i}$.

Let $i \in [n]$ such that $|\mC^i| \geq |\mC| / n$. Assume towards a contradiction that $\smash{|\mC^i| > |\mC| / n}$. Let $\smash{\mC_1 := \mC^i|_{S_1} \subseteq \mbS_{n-1}}$. Since $\smash{\dk(\mC_1) \geq \dk(\mC) \geq \binom{t}{2}}$, by  \cref{thm:Singletonbound2} we must have $\smash{|\mC_1| \leq \frac{(n-1)!}{t!}}$. However, this implies $\smash{|\mC| < n|\mC_1| \leq \frac{n!}{t!}}$, contradicting the fact that $\mC$ is $t$-balanced. Therefore  $|\mC^i| = |\mC| / n$ and $|\mC^j| \leq |\mC|/n$ for all $1 \leq j \leq n$. 
Furthermore, because $|\mC| = \sum_{i=1}^n |\mC^i| \leq \sum_{i=1}^n |\mC|/n$,  we must have that $|\mC^i| = |\mC|/n$ for all $1 \leq i \leq n$.

The fact that, for all $1 \leq i \leq n$, the code $\mC^i \subseteq \mbS_{n-1}$ is $t$-balanced is a direct consequence of $\smash{\dk(\mC^i|_{S_1}) \geq \dk(\mC) > \binom{t}{2}}$ and  $\smash{|\mC^i|_{S_1}| = \frac{|\mC|}{n} = \frac{(n-1)!}{t!}}$.
\end{proof}

As the next result shows, \cref{lem:optpunct} has strong implications for the combinatorial structure of $t$-balanced codes. The following result will be crucial for establishing the non-existence of $t$-balanced codes for $t \geq 3$. 

\begin{theorem}\label{thm:optstructure}
Let $X_{0,t} := \{t+1, \ldots, n\}$,  $X_{n-t,t} := \{1, \ldots, n-t\}$, and $X_{s,t} := \{1, \ldots, s, t+s+1, \ldots,  n\}$ for $1 \leq s \leq n-t-1$.
If $\mC \subseteq \mbS_n$ is a $t$-balanced code, then
for all ordered sequence $\mcI := (i_1, \ldots, i_{n-t})$ with $1 \leq i_j \leq n$ and $i_j \neq i_l$ for $j \neq l$, and $X := X_{s,t}$ where $0 \leq s \leq n-t$, there exists a unique codeword  $\sigma_{\mcI, X} \in \mC$ such that,  for all $j \in X_{s,t}$,
$$\sigma_{\mcI, X}(j) = \begin{cases}i_j  & \textup{if } j \leq s,\\
i_{j-t} &\textup{if } j > s.
    \end{cases} $$
\end{theorem}

\begin{proof}
Let $\mcI = (i_1, \ldots, i_{n-t})$ be an ordered sequence such that $1 \leq i_j \leq n$ and $i_j \neq i_l$ if $j \neq l$ and $X_{s,t}$ for some $0 \leq s \leq n-t$.  By iterating \cref{lem:optpunct} for each coordinate of $\mcI$ it follows that there exist a unique codeword $\sigma \in \mC$ such that,
for all $j \in X_{s,t}$,
$$\sigma(j) = \begin{cases}i_j  & \textup{if } j \leq s,\\
i_{j-t} &\textup{if } j > s,
    \end{cases}  $$ 
    as desired.
\end{proof}

We illustrate \cref{thm:optstructure} with the following example.

\begin{example}
    Let $A_4 \subseteq \mbS_4$ be the subgroup of even permutations:
     $$A_4 := \{[1234], [1342], [1423], [2143], [2314], [2431],[3124], [3241], [3412],  [4132], [4213], [4321]\}.$$
     Note that $A_4$ is a $2$-balanced code. 
As in \cref{thm:optstructure}, let $\mcI =(3,1)$ and consider the set $X_{1,2} = \{1,4\}$. By \cref{thm:optstructure}, there exists a unique codeword $\sigma \in A_4$ such that $\sigma(1) = 3$ and $\sigma(4) = 1$. In fact, this codeword is $\sigma = [3241]$. More generally, for any $\sigma, \tau \in A_4$, we must have $(\sigma(1), \sigma(4)) \neq (\tau(1), \tau(4))$. Similarly, from \cref{thm:optstructure} we can conclude that for any $\sigma, \tau \in A_4$, where $\sigma \neq \tau$, we have that $(\sigma(1), \sigma(2)) \neq (\tau(1), \tau(2))$ and $(\sigma(3), \sigma(4)) \neq (\tau(3), \tau(4))$.
\end{example}

We next prove the following lemma, which is the final ingredient needed to establish the main result of this section.

\begin{lemma}\label{lem:emo}
    Let $s \in [n-1]$ and $\sigma \in \mbS_n$ such that $\sigma(1) = s+1$ and $\sigma(n) \geq  s$.
    We have 
    $$|\wk(\sigma)| \leq \binom{n-1}{2} + 1.$$
Furthermore, equality holds if and only if $\sigma(n) = s$ and $\sigma(i)>\sigma(i+1)$ for all $1< i < n$.
\end{lemma}

\begin{proof}
 Since $\sigma(1) = s+1$, by \cref{lem:supportpuncture} we have $\wk(\sigma) = \wk(\sigma|_{S_1}) + s$. Let $\tau := \sigma|_{S_1}$. If $\sigma(n) = s$, then $\tau(n-1) = s$. If $\sigma(n) \geq s+2$, then $\tau(n-1) \geq s+1$. Thus $\tau(n-1) \geq s$ and by \cref{prop:weightpunct} we have
 \begin{equation}\label{eq:emo}
 \wk(\tau) = \wk(\tau|_{S_{n-1}}) + (n-1- \tau(n-1)) \leq \wk(\tau|_{S_{n-1}}) + n-1-s.\end{equation}
 Finally, note that $\tau|_{S_{n-1}} \in \mbS_{n-2}$, hence $\wk(\tau|_{S_{n-1}}) \leq \binom{n-2}{2}$. Putting everything together we get the following inequality
 \begin{align*}
     \wk(\sigma) &= \wk(\tau) + s \\
                 &\leq \wk(\tau|_{S_{n-1}}) + n-1-s + s\\
                 &\leq \binom{n-2}{2} + n-1\\
                 &= \binom{n-1}{2} +1.
\end{align*}

 This shows the first claim. Next, from \eqref{eq:emo} one can easily see that $\wk(\tau|_{S_n}) + (n-1- \tau(n-1)) = \wk(\tau|_{S_{n-1}}) + n-1-s$ if and only if $\tau(n-1) =s$, if and only if $\sigma(n) = s$. Finally, $\smash{\wk(\tau|_{S_{n-1}}) =  \binom{n-2}{2}}$ if and only if $\smash{\tau|_{S_{n-1}}(i) > \tau|_{S_{n-1}}(i+1)}$ for all $i \in [n-2]$, if and only if $\sigma(i) > \sigma(i+1)$ for all $1 < i < n$. This shows the second claim. 
\end{proof}

 We are now ready to prove the main result of this section, showing that $t$-balanced codes do not exist for $t \geq 3$. 

\begin{theorem}\label{thm:BOMB1}
   For all $n \in \NN$ and $ 3 \leq t \leq n$, there exists no $t$-balanced code $\mC \subseteq \mbS_n$. 
\end{theorem}

\begin{proof}
    Assume that a $t$-balanced code $\mC \subseteq \mbS_n$ exists, where $t \geq 3$. Without loss of generality, $\epsilon \in \mC$ by right-invariance; see~\eqref{RI}. By \cref{cor:optmindist} we have $\smash{\dk(\mC) = \binom{t}{2}+1}$. 
    From \cref{thm:optstructure}, for $0 \leq s \leq t$ there exists a unique codeword $\sigma_s \in \mC$ such that 
    $$\sigma_s(i) = \begin{cases} i & \textup{if } 1 \leq i \leq n-t-1,\\
    n-t+s & \textup{if }i = n-t. \end{cases}$$

    Our goal is to show that, for all $0 \leq s \leq t$, the entries $\sigma_s(i)$,
    for $n-t+1 \leq i \leq n$, are fully determined. 
    In turn, we will use said $\sigma_s$ to establish a contradiction on the minimum distance of our $t$-balanced code.

First consider $\sigma_0$. By \cref{thm:optstructure}, there exists a unique codeword $\sigma \in \mC$ such that $\sigma(i) = i$ for $1 \leq i \leq n-t$. Since $\epsilon \in \mC$ satisfies this property, it must be that $\sigma_0 = \epsilon$. 

We continue the proof by establishing two claims. 

\vspace{.1cm}

\textbf{Claim 1.} For $1 \leq s \leq t$ we have $\sigma_s(n) = n-t+s-1$ and $\sigma_s(i) > \sigma(i+1)$ for $n-t+1 \leq i \leq n-1$.

\vspace{.1cm}

\textit{Proof of Claim 1.} We proceed by induction on $s$ and start with $s=1$. Since $\smash{\dk(\mC) = \binom{t}{2}+1}$, we must have $\smash{\dk(\epsilon, \sigma_1) \geq \binom{t}{2}+1}$. 
Let $\smash{\epsilon' := \epsilon|_{[n]\backslash[n-t-1]}}$ and $\smash{\sigma_1' := \sigma_1|_{[n]\backslash[n-t-1]}}$.
Since $\sigma_1(i) = \epsilon(i)$ for $1 \leq i \leq n-t-1$, by applying \cref{prop:distpuncS1} iteratively we obtain

\begin{equation}\label{eq:claim1}
\dk(\epsilon, \sigma_1) = \dk(\epsilon', \sigma_1') = |I_{\sigma_1'}| \geq \binom{t}{2}+1.
\end{equation}

Furthermore, $\sigma_1' \in \mbS_{t+1}$, $\sigma_1'(1) = 2$, and $\sigma_1'(t+1) \geq 1$. Hence by \cref{lem:emo} we must have $|I_{\sigma_1'}| \leq \binom{t}{2}+1$.
Together with \eqref{eq:claim1}, this implies that $|I_{\sigma_1'}| = \binom{t}{2}+1$. Thus, once again by \cref{lem:emo}, $\sigma_1'(t+1) = 1$ and $\sigma_1'(i) > \sigma_1'(i+1)$ for all $2 \leq i \leq t$. The latter furthermore implies that $\sigma_1(n) = n-t$ and $\sigma_1(i) > \sigma_1(i+1)$ for all $n-t+1 \leq i \leq n-1$, proving the base case of Claim~1. 

Now suppose that the claim holds true for all $s \leq l-1$, for some $l \leq n$.
Consider $\sigma_{l}$. 
First note that for all $s \leq l$ we have $\sigma_s(i) = \sigma_l(i)$ for $i \leq n-t-1$, therefore by \cref{thm:optstructure} we have $\sigma_s(n) \neq \sigma_l(n)$.  By the induction hypothesis, $\sigma_s(n) = n-t+s-1$ for all $s < l$, hence $\sigma_l(n) > n-t+(l-1) -1 = n-t+l -2$.

Similarly to the base case, we reduce $\sigma_l$ to $\sigma_l' := \sigma_l|_{[n] \setminus [n-t-1]}$ and can easily show that $|I_{\sigma_l'}| = \binom{t}{2}+1$. By \cref{lem:emo}, we get that $\sigma_l'(t+1) = l$ and $\sigma_l'(i) > \sigma_l'(i+1)$ for $2 \leq i \leq t$. Thus $\sigma_l(n) = n-t+l-1$ and $\sigma_l(i) > \sigma_l(i+1)$ for $n-t+1 \leq i \leq n-1$,  establishing Claim 1.

\vspace{.1cm}

\vspace{.1cm}
\textbf{Claim 2:} $\dk(\sigma_n, \sigma_{n-2}) < \binom{t}{2}+1$ when $t \geq 3$. 

\vspace{.1cm}

\textit{Proof of Claim 2.} 
Let $\smash{\sigma'_j := \sigma_j|_{[n] \backslash [n-t-1]}}$ for $j \in \{n-1, n\}$.
Since $\sigma_n(i) = \sigma_{n-2}(i)$ for $1 \leq i \leq n-t-1$, by \cref{prop:distpuncS1} we have 
$\dk(\sigma_n, \sigma_{n-2}) = \dk(\sigma'_n, \sigma'_{n-2})$. Furthermore, by Claim 1 we know that
\begin{align*}
    \sigma'_{n-2} &= \begin{bmatrix}
        t-1 & t+1 & t & t-3 & \cdots & 1 & t-2
    \end{bmatrix},\\
    \sigma'_n &= \begin{bmatrix}
        t+1 & t-1 & t-2 & t-3 & \cdots & 1 & t
    \end{bmatrix}.
\end{align*}
Via a simple counting argument, we get that a minimum of at most $1 + (t+1 - 3) + (t+1-4)$ consecutive transpositions are needed to go from $\sigma'_n$ to $\sigma'_{n-2}$. Hence if $t \geq 3$,
$$\dk(\sigma'_n, \sigma'_{n-2}) \leq 2t-4 < \binom{t}{2} +1,$$
 where the second  inequality can be shown via a straightforward calculus argument. This proves Claim~2. 

The fact that both $\{\sigma_n, \sigma_{n-2}\} \subseteq \mC$ and $\dk(\sigma_n, \sigma_{n-2}) < \binom{t}{2}+1$ hold contradicts the minimum distance of $\mC$. Hence no $t$-balanced code $\mC$ can exist for $t \geq 3$. 
\end{proof}

The above result tells us that $t$-balanced codes exist if and only if $t=2$. We conclude the paper by 
characterizing $2$-balanced codes and showing that they must be the translate of the even permutation subgroup. 

\begin{theorem}\label{thm:2MDScharac}
    Let $\mC \leq \mbS_n$ be a $2$-balanced code and let $A_n \leq \mbS_n$ be the even permutation subgroup. Then 
    $$\mC = \sigma \circ A_n \textup{ for some } \sigma \in \mbS_n.$$
\end{theorem}

\begin{proof}
It suffices to show that if $\epsilon \in \mC$, then $\mC = A_n$. In fact, if $\epsilon \notin \mC$ then there exist $\sigma \in \mbS_n$ such that $ \epsilon \in \sigma \circ \mC$, which implies that $\sigma \circ \mC = A_n$. 

We proceed by induction on $n$. If $n=3$ and $\epsilon \in \mC$ it is easy to show via exhaustive search that $\mC$ is $2$-balanced if and only if $\mC = A_3$.
Now let $n \ge 4$ and assume that the claim is true for all $n' \leq n-1$. 
Let $\mC \subseteq \mbS_n$ be a $2$-balanced containing the identity, thus $|\mC| = n!/2$ and $\dk(\mC) = 2$, by \cref{cor:optmindist}. 
Let $\smash{S_1 := [n] \setminus \{1\}}$ and $\smash{\mC^i := \mC^{1,i}|_{S_1} \subseteq \mbS_{n-1}}$, where $1 \leq i \leq n$. By \cref{lem:optpunct},  $\mC^i $ is a $2$-balanced. Hence $\smash{\mC^i = \sigma_i \circ A_{n-1}}$ if $\epsilon \notin \mC^i$ for some $\sigma_i \in \mbS_{n-1}$,  or $\mC^i = A_{n-1}$  if $\epsilon \in \mC^i$. We next show that $\epsilon \in \mC^i$ if and only if $i$ is odd. 

First, consider the case $i = 1$. Since $\epsilon \in \mC$, $\mC^1$ is a $2$-balanced code containing the identity. Therefore, by the induction hypothesis, 
we must have that $\mC^1 = A_{n-1}$. Moreover, by \cref{prop:weightpunct}, $\wt(\sigma|_{S_1}) = \wt(\sigma)$ for all $\sigma \in \mC^1$. Thus if $\sigma \in \mC$ and $\sigma(1) = 1$, then $\sigma$ has even weight. Furthermore, since $|\mC_1| = (n-1)!/2$ and 
$$|\{\sigma \in \mbS_n \, : \, \sigma(1) = 1 \textup{ and } \wk(\sigma) \textup{ is even }\}| = (n-1)!/2,$$
it must be that $\sigma \in \mC^{1,1}$ if and only if $\sigma(1) = 1$ and $\sigma$ has even weight.

Now consider $\mC^2 \subseteq \mbS_{n-1}$. By induction hypothesis, $\mC^2$ consists of either all even weight permutations (if $\mC^2 = A_{n-1}$) or all odd weight permutations (if $\mC^2 = \sigma \circ A_{n-1} \neq A_{n-1}$). Suppose that $\mC^2 = A_{n-1}$. Then $\epsilon \in \mC^2$ which implies 
$$\tau := \begin{bmatrix}
    2 & 1 & 3 & \cdots & n
\end{bmatrix} \in \mC.$$ 
However, since $\epsilon \in \mC$ and $\dk(\epsilon, \tau) = 1$, we arrive at a contradiction on the minimum distance of $\mC$. Therefore $\mC^2$ consists of all odd weight permutations in $\mbS_{n-1}$. Furthermore, by \cref{prop:weightpunct} we get $\mC^{1,2} = \{\sigma \in \mbS_n \, : \, \sigma(1) = 2 \textup{ and } \wk(\sigma) \textup{ is even}\}$.

We now proceed by induction on $1 \le i \le n$ to show that 
$$\mC^i = \begin{cases}
    \{\sigma \in \mbS_{n-1} \, :  \wk(\sigma) \textup{ is even}\} & \textup{ if } i \textup{ is odd,}\\
    \{\sigma \in \mbS_{n-1} \, :  \wk(\sigma) \textup{ is odd}\} & \textup{ if } i \textup{ is even.}
\end{cases} $$
The base case of the induction ($i=1$) was shown previously. Assume that the statement holds true for all $j \leq i-1$.
Suppose first that $i$ is odd. Assume that $\mC^i$ consists of all odd weight elements. Then, $$\tau := \begin{bmatrix} i-1 & 1 & \cdots & n-1 \end{bmatrix} \in \mC^i,$$ since $\wt(\tau) = i-2,$ which is odd. This implies $$\alpha := \begin{bmatrix} i & i-1 & 1 & \cdots & i-2& i+1 & \cdots & n  \end{bmatrix} \in \mC.$$ By the induction hypothesis, we must have that $$\gamma := \begin{bmatrix}
    i-1 & i & 1 & \cdots & i-2 & i+1 & \cdots &n 
\end{bmatrix} \in \mC,$$ since $\gamma(1) = i-1$ and $\wt(\gamma) = 2(i-2)$, which is even. However, $\dk(\alpha, \gamma) =1$ gives a contradiction on the minimum distance of the code.
Thus $\mC^i$ must contain all odd weight elements of $\mbS_{n-1}$ and $\mC^{1,i} = \{\sigma \in \mbS_n \, : \sigma(1) = i \textup{ and } \wk(\sigma) \textup{ is even }\}$.

If $i$ is even, a similar proof holds. 
All of this shows that $\mC^{1,i}$ consists of only even weights for $1 \leq i \leq n$. Therefore $\mC$ contains only even weight permutations, which implies $\mC = A_n$. 
\end{proof}

Combining \cref{thm:BOMB1} and \cref{thm:2MDScharac} gives a full characterization of $t$-balanced codes for all $n \in \NN$ and $2 \leq t \leq n$. In addition, our result give a characterization of MDS codes as defined in \cite{jiangcorrecting10}, when $|C| = \smash{\frac{n!}{t!}}$. It remains however an open problem to characterize the codes that achieve the bound \eqref{eq:singl1}, i.e., codes such that $|\mC| > n!/t!$ and $\dk(\mC) = \binom{t}{2}$.

\section{Conclusion}

In this paper we proved a new bound on the maximum cardinality of codes under the Kendall-$\tau$ metric, with given length and minimum distance. We then classified the codes whose parameters achieve this bound with equality, called $t$-balanced codes. More precisely, we proved $t$-balanced codes only exist when $t=2$, in which case those codes must be translates of the even subgroup of the symmetric group of length $n$. Furthermore several open problems arose on the way, which we briefly outline.

\begin{enumerate}
    \item Do Kendall-$\tau$ balls of diameter $t$ in $\mbS_n$ larger than Kendall-$\tau$ cubes in $\mbS_n$ of the same diameter? This problem also relates to an open problem formulated in \cite{sarit} of characterizing the optimal anticodes with given diameter. 

    \item Give a full characterization of MDS codes. More precisely, determine for which parameter set there exist a code $\mC \subseteq \mbS_n$ such that $\smash{|\mC| > \frac{n!}{t!}}$ and $\smash{\dk(\mC) = \binom{t}{2}}$.
\end{enumerate}

\bibliographystyle{plain}
\bibliography{Kendall}

\end{document}